\documentclass[12pt]{amsart}
\usepackage{amssymb,amsmath}
\usepackage{qtree}
\usepackage{color}
\usepackage{graphicx}
\usepackage[vmargin=2.5cm, hmargin=2.5cm]{geometry}
\usepackage{enumerate}
\usepackage{ulem}
\usepackage{amssymb,amscd, graphics,color,latexsym,cancel}   
\usepackage[linktoc=all, pagebackref, hyperindex]{hyperref}%
\hypersetup{colorlinks,  citecolor=blue,  filecolor=blue,  linkcolor=red,  urlcolor=black}

\newtheorem{Theorem}{Theorem}[section]
\newtheorem{Lemma}[Theorem]{Lemma}
\newtheorem{Corollary}[Theorem]{Corollary}
\newtheorem{Proposition}[Theorem]{Proposition}

\theoremstyle{definition}
\newtheorem{Definition}[Theorem]{Definition}
\newtheorem{Remark}[Theorem]{Remark}
\newtheorem{Example}[Theorem]{Example}

\newenvironment{anumbered}%
{\begin{list}
        {\noindent\makebox[0mm][r]{\rm(\alph{enumi})}}
        {\leftmargin=6.3ex \usecounter{enumi}}
    }
    {\end{list}}

{\begin{list}
        {\noindent\makebox[0mm][r]{\rm(\roman{enumi})}}
        {\leftmargin=6.3ex \usecounter{enumi}}
    }
    {\end{list}}

\def\frk{\frak}               

\def\fp{{\frk p}}
\def\fq{{\frk q}}

\def\fm{{\frk m}}

\def\opn#1#2{\def#1{\operatorname{#2}}} 
\opn\chara{char}
\opn\length{\ell}
\opn\pd{pd}
\opn\rk{rk}
\opn\projdim{proj\,dim}
\opn\injdim{inj\,dim}
\opn\rank{rank}
\opn\depth{depth}
\opn\grade{grade}
\opn\height{height}
\opn\embdim{emb\,dim}
\opn\codim{codim}

\opn\Tr{Tr}
\opn\bigrank{big\,rank}
\opn\superheight{superheight}
\opn\lcm{lcm}
\opn\reg{reg} \opn\lreg{lreg} \opn\ini{in} \opn\lpd{lpd}
\opn\size{size} \opn\sdepth{sdepth}
\opn\link{link}\opn\fdepth{fdepth}\opn\lex{lex} \opn\Att{Att} \opn\E{E}
\opn\div{div} \opn\Div{Div} \opn\cl{cl} \opn\Cl{Cl} \opn{\Ann}{Ann} \opn{\CM}{CM}
\opn\Spec{Spec} \opn\Supp{Supp} \opn\supp{supp} \opn\Sing{Sing}
\opn\Ass{Ass}   \opn\Assh{Assh}  \opn\height{height}  \opn\Min{Min}\opn\Mon{Mon}
\opn\Ann{Ann} \opn\Rad{Rad} \opn\Soc{Soc}
\opn\Im{Im} \opn\Ker{Ker} \opn\Coker{Coker} \opn\Am{Am}
\opn\Hom{Hom} \opn\Tor{Tor} \opn\Ext{Ext} \opn\End{End}
\opn\Aut{Aut} \opn\id{id}

\opn\nat{nat}
\opn\pff{pf}
\opn\Pf{Pf} \opn\GL{GL} \opn\SL{SL} \opn\mod{mod} \opn\ord{ord}
\opn\Gin{Gin} \opn\Hilb{Hilb}\opn\sort{sort}
\opn\L{L}
\opn\aff{aff} \opn
\con{conv} \opn\relint{relint} \opn\st{st}
\opn\lk{lk} \opn\cn{cn} \opn\core{core} \opn\vol{vol}
\opn\link{link} \opn\star{star}\opn\lex{lex}\opn\set{set}
\opn\gr{gr}

\def\pot#1#2{#1[\kern-0.28ex[#2]\kern-0.28ex]}

\opn\dirlim{\underrightarrow{\lim}}
\opn\inivlim{\underleftarrow{\lim}}
\let\sect=\cap

\let\Union=\bigcup

\def\Implies{\ifmmode\Longrightarrow \else
    \unskip${}\Longrightarrow{}$\ignorespaces\fi}
\def\implies{\ifmmode\Rightarrow \else
    \unskip${}\Rightarrow{}$\ignorespaces\fi}
\def\iff{\ifmmode\Longleftrightarrow \else
    \unskip${}\Longleftrightarrow{}$\ignorespaces\fi}

\let\:=\colon

\newcommand{\excise}[1]{}
\begin{document}

\title{On the stable property of projective dimension }

\author{Somayeh Bandari}
\address{Somayeh Bandari,  Department of Mathematics, Buein Zahra Technical University,  Buein Zahra, Qazvin,
Iran and School of Mathematics, Institute for Research in
Fundamental Sciences (IPM) P. O. Box: 19395-5746, Tehran, Iran.}
\email{somayeh.bandari@yahoo.com}

\author{Raheleh Jafari}
\address{Raheleh Jafari, Mosaheb Institute of Mathematics, Kharazmi University, Tehran, Iran}
\email{rjafari@ipm.ir}

\thanks {The first author  was in part supported by a grant from IPM (No.
95130020)}



\maketitle

\begin{abstract}
We introduce the concept of monomial ideals with  stable projective
dimension, as a generalization of the Cohen-Macaulay property.
Indeed, we study the class of monomial ideals $I$, whose projective
dimension is stable under
 monomial localizations at monomial prime ideals $\fp$,  with $\height \fp\geq \pd
 S/I$. We study the relations between this  property and other  sorts of  Cohen-Macaulayness.
 Finally, we  characterize some classes of polymatroidal
ideals with stable projective dimension.
\end{abstract}

\section*{Introduction}
Throughout the paper, $S=K[x_1,\ldots,x_n]$ denotes the polynomial
ring in $n$ indeterminates over an arbitrary field $K$. Let  $M$ be
a  finitely generated graded $S$-module. Then
\begin{equation}\label{1}
\{\fp\in \Supp(M) \;: \; \pd M=\pd M_\fp\}\subseteq \{\fp\in\Supp(M)  \;: \; \height \fp\geq \pd M\},
\end{equation}
where $\pd(-)$ denotes the projective dimension
(cf.~Lemma~\ref{lem1}). If $M$ has positive depth, then the equality
holds in (\ref{1}), precisely when  $M$ is Cohen-Macaulay
(cf.~Proposition~\ref{Thm8}). If we replace  $M$ with $S/I$, for a
monomial ideal $I$ in $S$, it is natural to consider only monomial
prime ideals in $\Supp(S/I)=V(I)$. The monomial localization  of a
monomial ideal $I\subset S$ at a monomial prime ideal $\fp$, is the
monomial ideal $I(\fp)\subseteq S(\fp)=K[x_i \ : \ x_i\in\fp]$
obtained from $I$ by setting  $x_i=1$, for all variables $x_i \notin
\fp$. Note that
    $\pd(S(\fp)/I(\fp))=\pd(S/I)_\fp$~(cf.~Remark~\ref{Rem}). Let
$V^*(I)$ denote the set of monomial prime ideals $\fp$ in $V(I)$.
Example~\ref{exam2}, provides a monomial ideal $I$, with
$\depth(S/I)>0$, that satisfies
\begin{eqnarray}\label{1*}
\{\fp\in V^*(I)\;: \; \pd S/I=\pd(S(\fp)/I(\fp))\}=\{\fp\in V^*(I)  \;: \; \height \fp\geq \pd S/I\},
\end{eqnarray}
but $S/I$ is not Cohen-Macaulay.
The main interest in this paper, is to study the class of monomial ideals $I$, with property (\ref{1*}).
In other words, we study the class of monomial ideals $I$ whose projective dimension is stable under
 monomial localizations at monomial prime ideals $\fp$,  with $\height \fp\geq \pd S/I$. In this case, $I$
  is said to have \textit{stable projective dimension} (cf.~Definition~\ref{def}).
This class contains all monomial ideals $I$, where $S/I$ is a
Cohen-Macaulay ring (cf.~Proposition~\ref{CM-stb}). Indeed the set
of monomial ideals with Cohen-Macaulay quotient rings, is precisely
the set of monomial ideals $I$ with stable projective dimension such
that $\Ass(S/I)=\Min(I)$ (cf.~Proposition~\ref{min}).
 Taking into account that the monomial localization is a simple operation which
  turns out a monomial ideal in a smaller polynomial ring, the class of ideals with some
   stable algebraic properties under monomial localizations, have significant role in monomial algebra.

Beginning in Section~\ref{stb}, we study $\Ext^{n-t}(M,S)$, where
$M$ is a finitely generated graded $S$-module with $\depth(M)=t$.
Expressing $\Supp(\Ext^{n-t}(M,S))$, as  the set of prime ideals
$\fp\in\Supp(M)$ with $\pd M_\fp=\pd M$,  we specify when the
equality holds in~(\ref{1}). Then we discuss the relations between
the property of having stable projective dimension, and some weaker
sorts of Cohen-Macaulayness. Several examples are included to
indicate
 that the set of monomial ideals with stable projective dimension, is disjoint from the set
 of monomial ideals $I$, where $S/I$ is generalized Cohen-Macaulay or sequentially Cohen-Macaulay,
 and non of them contains the other. More over, it is shown that a monomial ideal $I$, where $S/I$ is a generalized
 Cohen-Macaulay ring,  has stable projective dimension if and only if, either $S/I$ is Cohen-Macaulay
    or $\depth S/I=0$ (cf.~Proposition~\ref{gen}).

Section~\ref{pol}, is devoted to characterizing some classes of
polymatroidal ideals who have stable projective dimension. We show
that polymatroidal ideals generated in degree $2$ with at least one
pure power and Veronese type ideals have stability property of
projective dimension. We also characterize transversal polymatroidal
ideals with this property.

Several explicit examples are provided  along the  paper to
illustrate the  property of having stable projective dimension. Many
of them have been computed by using  CoCoA   \cite{cocoa}.

\section{Stability of projective dimension under localizations}\label{stb}
In this paper, we consider the natural grading on the polynomial
ring $S=k[x_1,\ldots,x_n]$ and   all $S$-modules are graded. We use
several times the Auslander-Buchsbaum formula
\cite[A.4.3]{Herzog-Hibi-2011}, saying that
\[\pd(M)+\depth(M)=n,\]
 for all finitely generated graded $S$-modules $M$.
We start with the following easy lemma, that is a  crucial point in
our approach.

\begin{Lemma}\label{lem1}
    Let $M$ be a  finitely generated graded $S$-module with
    $\depth M=t$. Then
    \begin{eqnarray}
    \Supp(\Ext^{n-t}(M,S))&=&\{\fp\in \Supp(M) \;: \; \pd M=\pd M_\fp\}\\
    &\subseteq&\{\fp\in\Supp(M)  \;: \; \height \fp\geq \pd M\}.\label{4}
    \end{eqnarray}
\end{Lemma}
\begin{proof}
    By the Auslander-Buchsbaum formula, $\pd M=n-t$. Since,
    $\pd M_\fp\leq\pd M$ for any prime ideal $\fp$ of $S$, $\Ext^{n-t}(M_\fp,S_\fp)\neq 0$
     precisely when  $\pd M_\fp=\pd M$. This provides the first equality. The inequality,
     follows from the fact that $\pd M_\fp=\height \fp-\depth M_\fp$.
\end{proof}

The following result, specifies when the equality holds in  (\ref{4}).

\begin{Proposition}\label{Thm8}
Let $M$ be a nonzero finitely generated graded $S$-module. Then the
following statements are equivalent.
    \begin{anumbered}
        \item  $\pd M=\pd M_\fp$, for all $\fp\in\Supp(M)$ with $\height \fp\geq \pd M$.
        \item Either $M$ is Cohen-Macaulay or $\depth M=0$.
    \end{anumbered}
    If the above conditions hold, then  $\dim \Ext^{n-t}(M,S)=t$.
\end{Proposition}
\begin{proof}
    (a)\implies(b).  By Lemma~\ref{lem1},    $\{\fp\in\Supp(M) \ ; \ \height
    \fp=\pd M \}=\Min(\Ext^{n-t}(M,S))$,  is  a finite set.
    Hence either $\pd(M)=\height(\fp)$ for all minimal primes $\fp$ of $M$, or $\pd(M)=n$, from Ratliff's weak existence
    theorem \cite[Theorem 31.2]{Matsumura}. Consequently $\depth M=\dim M$  or $\depth(M)=0$.

    (b)\implies(a). If $\depth M=0$, then $\pd M=n$ and so the graded maximal ideal of $S$ is the only prime ideal
    of height greater than or equal to $\pd M$. Assume that  $M$ is Cohen-Macaulay. Let $I=\Ann M$. Then  $M_\fp$ is
    Cohen-Macaulay, for all $\fp\in \Supp M$ and
    \[
    \pd M_\fp=\height\fp-\dim M_\fp=\height I_\fp=\height I =\pd M,\]
    which implies (a).     For the last statement, let $J=\Ann(\Ext^{n-t}(M,S))$ and $I=\Ann M$. Then
    $$I\subseteq J \;\text{and}\; \height I\leq\pd M.$$
    If
    $\pd M<\height J$, then there exists $\fp\in\Supp(M)$ such that
    \[\pd M\leq\height \fp<\height J.\] By stability property of
    projective dimension, $\fp\in\Supp\Ext^{n-t}(M,S)$, consequently  $J\subseteq\fp$,
    a contradiction. Hence $\pd M\geq\height J$ and so
    \[
    \pd M=n-\depth M\leq n-\dim \Ext^{n-t}(M,S) =\height J\leq\pd M,\]
    which completes the proof.
\end{proof}

\begin{Corollary}\label{ECM}
   Let $M$ be a nonzero finitely generated graded $S$-module. If  $\pd M=\pd M_\fp$ for all $\fp\in\Supp(M)$ with $\height \fp\geq
    \pd M$, then $\Ext^{n-t}(M,S)$ is a Cohen-Macaulay module of
    dimension $t=\depth(M)$.
\end{Corollary}

\begin{proof}
    We have that $M$ is Cohen-Macaulay or $\depth(M)=0$, by
    Proposition~\ref{Thm8}. If  $M$ is Cohen-Macaulay of $\dim(M)=t$, then
    $\Ext^{n-t}(M,S)$ is a Cohen-Macaulay module of dimension $t$,
    \cite[Theorem~1.4]{Herzog-Sbarra}. Now, if $\depth(M)=0$, then $\pd
    M=n$. Hence  by Lemma \ref{lem1},
    $\Supp(\Ext^{n-t}(M,S))=\{\fm\}$, and  $\dim(\Ext^{n-t}(M,S))=0$.
\end{proof}

The monomial localization  of a monomial ideal $I\subset S$ with respect to a monomial prime ideal $\fp$ is
 the monomial ideal $I(\fp)\subset K[x_i : x_i\in\fp]$ which is
obtained from $I$ by substituting the variables $x_i\not \in \fp$ by
$1$.

\begin{Remark}\label{Rem}
    Let $\fp$ be a monomial prime ideal containing the monomial ideal $I$.
    Then  $I(\fp)$ is the unique monomial ideal with the
    property that $I(\fp)S_\fp=IS_\fp$.
   For a polynomial $f\in S$, let $f(\fp)$ denote the polynomial in $S(\fp)$,
    obtained from $f$ by substituting the variables $x_i\not \in \fp$ by
        $1$. If $I$ is an  ideal in $S$ generated by monomials $u_1,\ldots,u_r$, then $I(\fp)$ is
         the monomial ideal generated by $u_1(\fp),\ldots,u_r(\fp)$ in $S(\fp)$. Note that $u_i=u_i(\fp)v_i$,
         for some $v_i\notin\fp$. As $v_i$ is an invertible element in $S_\fp$, it is not difficult to see
         that $\depth(S/I)_{\fp}=\depth S(\fp)/I(\fp)$ and $\dim (S/I)_{\fp}=\dim S(\fp)/I(\fp)$.
    In particular $\pd(S/I)_{\fp}=\pd S(\fp)/I(\fp)$.
\end{Remark}

The following example, shows that the  converse of Corollary~\ref{ECM} is not true.

\begin{Example}\label{coun1}
   Let $S=K[x,y,z,u]$, $I=(x,y)\cap(y,z,u)$ and  $M=S/I$.  Then $\depth(M)=1$ and  $\pd M=3$.
    Note that $\Supp(\Ext^3(M,S))\subseteq\Supp(M)=V(I)$ and  $(x,y)$ is the only  prime ideal of height 2, in $V(I)$.
   Since $\pd(M_{(x,y)})=2$ and $\pd(M_{(y,z,u)})=3$, we have
   $(x,y)\notin\Supp(\Ext^3(M,S))$ and $(y,z,u)\in\Supp(\Ext^3(M,S))$. In particular $\dim\Ext^3(M,S)=\dim S/(y,z,u)=1=\depth M$.
We may also see that   $\depth\Ext^3(M,S)=1$, using computations by
CoCoA   \cite{cocoa}.
       \end{Example}

If $\pd M=\pd M_\fp$ only for  monomial prime ideals $\fp$,
with $\height(\fp)\geq\pd(M)$, that is a finite
set, then we don't necessarily get the statement (b) in
Proposition~\ref{Thm8}. Here is an example:

\begin{Example}\label{exam2}
    Let $S=K[x,y,z]$ and $M=S/I$, where $I$ is the monomial ideal
    $I=x(xy,xz,yz)=(x)\cap(x^2,y)\cap(x^2,z)\cap(y,z)$.
    Then $\dim M=2$,  $\depth M=1$ and $\pd(M)=2$. Let $\fp\neq\fm$ be a monomial prime ideal in $\Supp(M)$, with $\height(\fp)\geq 2$.
    Then $\fp\in\Ass(M)$ and so $\pd(M_\fp)=2-\depth(M_\fp)=2$. But it does not hold for the non-monomial prime $\fq=(x,y+z)$.
    Since $M_\fq=S_\fq/(x^2)S_\fq=(S/(x^2))_\fq$ is Cohen-Macaulay of dimension one, we have $\pd(M_\fq)=\height(\fq)-\depth M_\fq=2-1=1$.
\end{Example}

In the sequel, we consider  stability property for monomial ideals with
localization in monomial prime ideals. In other words, we are interested to study the class of $S$-modules, which intersecting the sets in (\ref{4})
 with the set of monomial ideals, the equality holds.


\begin{Definition}\label{def}
    A monomial ideal $I\subset S$  \textit{has stable projective dimension} if
    $\pd S/I=\pd S(\fp)/I(\fp)$ for all monomial prime ideals $\fp\in V^*(I)$  with $\height \fp\geq \pd S/I$.
\end{Definition}

Note that a monomial ideal $I$ has stable projective dimension if
and only if
\[
    V^*(I)\cap\Supp(\Ext^{n-t}(S/I,S))=\{\fp\in V^*(I) \ : \ \height(\fp)\geq\pd S/I\},
    \]
where $t=\depth S/I$.

\begin{Proposition}\label{CM-stb}
\label{immediate} Let $I\subset S$ be a monomial ideal. Then the
following conditions are equivalent:
\begin{anumbered}
       \item $S/I$ is Cohen--Macaulay.
       \item $\pd S/I=\pd S(P)/I(P)$ for all $P\in V^*(I)$.
\end{anumbered}
In particular, if $S/I$ is Cohen-Macaulay, then $I$ has stable
projective dimension.
\end{Proposition}

\begin{proof}
(a)\implies (b): Since $S/I$ is Cohen-Macaulay, then $S(P)/I(P)$ is
Cohen--Macaulay for all $P\in V^*(I)$. Therefore
\[
\pd S(P)/I(P)=\height I(P)=\height I =\pd S/I
\]
for $P\in V^*(I)$, as desired.

(b)\implies (a): We choose  $P\in V^*(I)$ such that $\height
P=n-\dim S/I$. Then we get
\[
n-\depth S/I=\pd S/I=\pd S(P)/I(P)\leq\height P=n-\dim S/I.
\]
It follows that $\depth S/I=\dim S/I$. In other words, $S/I$ is
Cohen--Macaulay.

\end{proof}

\begin{Corollary}\label{height1}
    All monomial ideals in $S$ of  height $n-1$, have stable projective dimension.
\end{Corollary}

\begin{proof}
    Let $I\subset S$ be a monomial ideal of height
    $n-1$. So either $\depth S/I=0$ or  $\depth S/I\geq
    1=\dim S/I$. Hence either $\depth S/I=0$ or $S/I$ is
    Cohen-Macaulay. So by proposition \ref{Thm8}, $I$ has stable projective dimension.
\end{proof}

For a monomial ideal $I$, the ring $S/I$ is called {\it generalized
Cohen-Macaulay}, if $(S/I)_{\fp/I}$ is Cohen-Macaulay for all  prime
ideals $\fp\in V(I)\setminus\{\fm\}$, and $I$ is equidimensional,
i.e. $\dim(S/\fp)=\dim(S/I)$ for all  $\fp\in\Min I$. Let $I$ be a
monomial ideal of $S$, then the ring $S/I$ is generalized
Cohen-Macaulay precisely when $I$ is equidimensional and the ring
$S(\fp)/I(\fp)$ is Cohen-Macaulay for all monomial prime ideals
$\fp\in V^*(I)$ with $\fp\neq\fm$ (cf.~\cite[Lemma~4.1]{BJ}).

    \begin{Example}
    Let $I$ be the ideal in Example~\ref{exam2}. Then  $I$ has stable projective dimension,
    but $S/I$ is not generalized Cohen-Macaulay, because $I$ is not equidimensional.
    \end{Example}
\begin{Example}\label{examgen}
    Let $I=(x_1,x_2)\cap(x_3,x_4)$ be an ideal in $S=K[x_1,\ldots,x_4]$. Then $\pd S/I=3$,
    while $\pd S(\fp)/I(\fp)=\pd(K[x_1,x_2,x_3]/(x_1,x_2))=2$, for $\fp=(x_1,x_2,x_3)$.
    Therefore, $I$ does not have stable projective dimension. Since, any monomial prime ideal
    $\fp\in V^*(I)\setminus\{\fm\}$ contains only one of $(x_1,x_2)$ or $(x_3,x_4)$, $S(\fp)/I(\fp)$ is Cohen-Macaulay.
    In particular, $S/I$ is generalized Cohen-Macaulay.
\end{Example}

The above two examples, show that the class of monomial ideals with
stable projective dimension, and the class of monomial ideals $I$
such that $S/I$ is a generalized Cohen-Macaulay ring, are disjoint.
The following result, specifies   the intersection of these classes.
Recall that for  a finitely generated $S$-module $M$,
$\Assh(M)=\{\fp\in\Ass(M) \ : \ \dim S/\fp=\dim M\}$.

\begin{Proposition}\label{gen}
    Let $I\subset S=K[x_1,\ldots,x_n]$ be a generalized Cohen-Macaulay.
    Then the following conditions are equivalent:
    \begin{anumbered}
        \item $I$ has stable projective dimension.
        \item Either $S/I$ is Cohen-Macaulay or $\depth S/I=0$.
    \end{anumbered}
\end{Proposition}

\begin{proof}
    (a)\implies (b): Let $I$ have stable projective dimension with
    $\depth S/I\neq 0$ which  is not Cohen-Macaulay. Since $I$ is not
    Cohen-Macaulay, it follows that $\pd S/I>\height I$. Let $\fp'\in
    \Assh S/I$ and $\fp=\fp'+\fq\in V^*(I)$ be a monomial prime ideal with
    $\height \fp=\pd S/I$. Since $\height \fp=\pd S/I=n-\depth S/I<n$ , it
    follows that $\fp\neq \fm$. Hence since $I$ is generalized
    Cohen-Macaulay, we have that $S(\fp)/I(\fp)$ is Cohen-Macaulay.
    Therefore
    $$\pd S(\fp)/I(\fp)=\height I(\fp)\leq \height I=\height \fp'<\height \fp=\pd
    S/I.$$ So $I$ does not have stable projective dimension.

    (b) \implies (a) is obvious by Proposition \ref{Thm8}.
\end{proof}

\begin{Remark}
    For a  monomial ideal $I$ with $\height(I)=n-1$, the ring $S/I$ is a one dimensional
     generalized Cohen-Macaulay ring. Proposition~\ref{gen} interprets  the argument of the proof of  Corollary~\ref{height1},  in a more general content.
\end{Remark}

 For a finitely generated $S$-module $M$,  of dimension
$d$, let $M_i$ denote the largest submodule of $M$ such that $\dim
M_i\leq i$, for $0\leq i<d$. Note that $M_{i-1}\subseteq M_i$. The
increasing filtration $\{M_i\}_{i=0}^d$, is called the dimension
filtration. Note that $\Ass M_i/M_{i-1}=\{\fp\in\Ass M \ : \ \dim
S/\fp=i\}$ (cf.~\cite[Corollary~2.3]{Sc}).  Therefore, $M_i/M_{i-1}$
is either zero or an $i$-dimensional module.
  $M$ is called \textit{sequentially Cohen-Macaulay}, if  $M_i/M_{i-1}$ is  a Cohen-Macaulay $S$-module for $i=1,\ldots,d$.

\begin{Remark}\label{seq}
 Let $M$ be a finitely generated $S$-module of dimension
$d$. If $\Ass(M)=\Assh(M)$, then $M_i=0$, for $i=0,\ldots,d-1$ and
$M_d=M$. In particular, $M$ is Cohen-Macaulay if and only if $M$ is
sequentially Cohen-Macaulay and $\Ass(M)=\Assh(M)$.
\end{Remark}

\begin{Example}\label{examseq}
    Let $S=K[x_1,\ldots,x_4]$ and $I=(x_1x_2,x_2x_3,x_3x_4,x_1x_4,x_1x_3)=(x_1,x_3)\cap(x_1,x_2,x_4)\cap(x_2,x_3,x_4)$.
    Since $I$ is the edge ideal of a chordal graph, $S/I$ is sequentially Cohen-Macaulay, by \cite[Theorem~3.2]{FV}. $R/I$ is a
     two dimensional ring, which is not Cohen-Macaulay and does not have the maximal ideal in its associated prime ideals. Therefore,
      $\depth S/I=1$ and $\pd S/I=3$. Let $\fp=(x_1,x_2,x_3)$. Then $S(\fp)/I(\fp)=K[x_1,x_2,x_3]/(x_1,x_3)$ is of
       projective dimension $2$. Hence $I$ does not have stable projective
    dimension.
\end{Example}

 By definition, all monomial ideals $I$  with $\depth(S/I)=0$, have
stable projective dimensions. But $S/I$ is not necessarily
sequentially Cohen-Macaulay. Here is an example in the class of
monomial ideals, with positive depth.
\begin{Example}\label{Exx}
 Let $S=K[x_1,\dots,x_4]$ and $I=(x_1,x_4)\cap(x_2,x_3)\cap(x_1,x_2,x_3)^2\cap(x_1,x_2,x_4)^2\cap(x_1,x_3,x_4)^2\cap(x_2,x_3,x_4)^2$.
 Then $\pd S/I=3$. As all monomial prime ideals of height 3, are among the associated radical ideals, $\depth S(\fp)/I(\fp)=0$,
 which implies $\pd S(\fp)/I(\fp)=3$, for all $\fp$ with $\height \fp=3$. Therefore, $I$ has stable projective dimension. Note that,
  $\sqrt{I}^\vee=((x_1,x_4)\cap(x_2,x_3))^\vee=(x_1x_4,x_2x_3)$, does not have linear resolution.
  Therefore, $S/\sqrt{I}$ is not sequentially Cohen-Macaulay, by \cite[Theorem~3.1]{FV}. Now, \cite[Theorem~2.6]{HTT},
  implies that $S/I$ is not sequentially Cohen-Macaulay.
\end{Example}

The above examples, show that the class of monomial ideals $I$,
where $S/I$ is sequentially Cohen-Macaulay and the class of ideals
with stable projective dimension, are disjoint and non of them
contains the other.

\begin{Proposition}\label{min} Let $I\subset S$ be a  monomial ideal
    with $\Ass(S/I)=\Min(I)$. Then $I$ has stable projective dimension
    if and only if $S/I$ is Cohen-macaulay.
\end{Proposition}

\begin{proof}
    If $S/I$ is  Cohen-Macaulay, then obviously $I$ has stable projective
    dimension. Now if $S/I$ is not Cohen-Macaulay, we show that $I$ does
    not have stable projective dimension. Since $S/I$ is not
    Cohen-Macaulay, it follows that $\pd S/I>\height I$. Let $\fp'\in
    \Assh S/I$ and $\fp=\fp'+\fq$ be a monomial prime ideal with $\height
    \fp=\pd S/I$. Since $\Ass(S/I)=\Min(I)$ and $\fp'\subsetneq \fp$, it
    follows that $\fp\notin \Ass S/I$, so $\depth S(\fp)/I(\fp)\neq 0$.
    Therefore
    $$\pd S(\fp)/I(\fp)=\height \fp-\depth S(\fp)/I(\fp)=\pd S/I-\depth
    S(\fp)/I(\fp)< \pd S/I.$$ So $I$ does not have stable projective
    dimension.
\end{proof}

\begin{Remark}\label{1.8}
    The above result implies  that for squarefree  or unmixed monomial
    ideals (ideals $I$ with $\Ass(S/I)=\Assh(S/I)$), stable projective dimension property of $I$, is equivalent with
    Cohen-Macaulay property of $S/I$.  Example~\ref{equi},
     provides an equidimensional monomial ideal $I$ with stable projective dimension, such that $S/I$  is not Cohen-Macaulay.
\end{Remark}

The following result is a consequence of Proposition~\ref{min} and
Remark~\ref{seq}.
\begin{Corollary}\label{St-CM-SC}
Let $I\subset S$ be a monomial ideal, such that
$\Ass(S/I)=\Assh(S/I)$. Then the  following statement are
equivalent.
\begin{anumbered}
    \item $I$ has stable projective dimension.
\item $S/I$ is Cohen-Macaulay.
    \item $S/I$ is sequentially Cohen-Macaulay.
\end{anumbered}
\end{Corollary}

In the above corollary, (b) always implies (a) and (c), without any
assumption on $\Ass(S/I)$.
 The following example, shows that the condition  $\Ass(S/I)=\Assh(S/I)$ is necessary for the inverse.

\begin{Example}
Let $S=K[x_1,x_2,x_3]$ and
$I=(x_1)\cap(x_2,x_3)\cap(x_1^2,x_2)\cap(x_1^2,x_3)$. Then $\pd
S/I=2$. As all monomial prime ideals $\fp$ with $\height(\fp)=2$,
belong to the associated prime ideals of $I$, $\pd S(\fp)/I(\fp)=2$.
In particular, $I$ has stable projective dimension. Since $\dim
S=3$, the ideal $I$ is pretty clean by  \cite[Theorem~1.10]{S}.
Therefore, $S/I$ is sequentially Cohen-Macaulay by
\cite[Corollary~4.3]{HP}.
\end{Example}

\section{Polymatroidal ideals with stable projective dimension}\label{pol}
In this section we  characterize some special classes of
polymatroidal ideals with stable projective dimension. Let $I\subset
S=K[x_1,\ldots,x_n]$ be a monomial ideal generated in a single
degree and  $G(I)$ be the unique minimal set of monomial generators
of $I$. Then $I$ is said to be polymatroidal, if for any two
elements $u,v\in G(I)$ such that $\deg_{x_i}(u)> \deg_{x_i}(v)$
there exists an index $j$ with $\deg_{x_j}(u)< \deg_{x_j}(v)$ such
that $x_j(u/x_i)\in I$. In the sequel a monomial prime ideal $\fp$
will be denoted by $\fp_{\{i_1,\ldots,i_t\}}$, where
$\{i_1,\ldots,i_t\}=[n]\setminus\{x_i ; x_i\in \fp\}$.

\begin{Proposition}\label{degree 2}
    Let $I\subset S=K[x_1,\ldots,x_n]$ be a polymatroidal generated in
    degree 2 with at least one pure power $x_i^2\in I$, for some $1\leq i\leq n$. Then $I$ has stable projective dimension.
\end{Proposition}

\begin{proof}
    Without loss of generality, we may assume that $I$ is fully
    supported.
    By \cite[Proposition 2.1 (e)]{BaHe}, after relabeling of the
    variables there exists integer $1\leq k\leq n$ such that
    $$I=(x_1,\ldots,x_k)(x_1,\ldots,x_n)+ J,$$
    where $J$ is a  squarefree monomial ideal in the variables
    $x_{k+1},\ldots,x_n$. Set $\fq:=(x_1^2,x_2,\ldots,x_n)$. Obviously $\fq$ is one of the irreducible component
    of minimal primary decomposition of $I$. Hence
    $\fm=\sqrt{\fq}\in \Ass(S/I)$, and $\depth
    S/I=0$.
   As $\fm$ is the only monomial ideal with height  equal to $n=\pd(S/I)$,
 $I$ has stable
    projective dimension.
\end{proof}
The following example shows that the condition of
having at least one pure power is essential.
\begin{Example}
Let $I=(x_1x_2,x_1x_3)\subset K[x_1,x_2,x_3]$. $I$ is matroidal
ideal generated in  degree 2. As $S/I$ is not Cohen-Macaulay, $I$
does not have stable  stable projective dimension by
Remark~\ref{1.8}.
\end{Example}

 An interesting special case of polymatroidal ideals is that of
ideals of Veronese type. Given  positive integers $d, a_1,\ldots,
a_n$ such that $1\leq a_i\leq d$. We let
$I_{(d;a_1,\ldots,a_n)}\subset  S=K[x_1,\ldots,x_n]$ be the monomial
ideal generated by the monomials $u\in S$ of degree $d$ satisfying
$\deg_{x_i}(u)\leq a_i$ for  all $i=1,\ldots,n$. Monomial ideals of
this type are called ideals of  {\it Veronese type}.

For the proof of the next theorem, we need the following results.

\begin{Lemma} \label{knownformula}\cite[Theorem 3.3]{HH2}    Let $I=I_{(d;a_1,\ldots,a_n)}$ with $1\leq a_i\leq d$  for all $i$. Then
    \[
    \pd S/I=\min\{n, \sum_{i=1}^na_i-d+1\}.
    \]
\end{Lemma}

\begin{Lemma}\label{veronese} Let  $J$ be a Veronese type ideal
of $S$ and $\fp$ be a monomial prime ideal of $S$ with $\height
\fp\geq \pd S/J$. Then $\fp\in V^*(J)$.
\end{Lemma}

\begin{proof}
Let $J=I_{(d;a_1,\ldots,a_n)}$ with $1\leq a_i\leq d$ for all $i$,
and  $\fp$ be a monomial prime ideal of $S$ with $\height \fp\geq
\pd S/J$. If $\pd S/J=n$, then $\fm\in\Ass(S/J)$ and $\height \fp
\geq n$. Hence $\fp=\fm\in V^*(J)$. Now let $\pd(S/J)=\sum_{i=1}^n
a_i-d+1$. If $\sum_{i=1}^n a_i=d$, then $J=(x_1^{a_1}\cdots
x_n^{a_n})$. Since $a_i\geq 1$ for all $i\in[n]$, $\fp\in V^*(J)$.
Now assume that $\sum_{i=1}^n a_i>d$. We want to show that
$\fp\in\Ass(S/J^k)$, where $k=\height \fp-1$,
 then obviously $\fp\in V^*(J)$. Since $\height
\fp\geq \sum_{i=1}^n a_i-d+1$ and $\sum_{i=1}^n a_i>d$, we have
that $k\geq 1$. By \cite[Lemma 5.1]{HRV},
$J^k=I_{(kd;ka_1,\ldots,ka_n)}$. Since $\sum_{i=1}^n a_i>d$, it
follows that $k(\sum_{i=1}^n a_i-d)\geq k=\height\fp-1$. Hence
$k(\sum_{i=1}^n a_i)\geq kd-1+\height \fp$. On the other hand,
Since $a_i\geq 1$, for all $i$, then
 $$\height\fp\geq \sum_{i=1}^n a_i-d+1=\sum_
{x_i\in \fp} a_i+\sum_{x_i\notin \fp} a_i-d+1\geq
\height\fp+\sum_{x_i\notin \fp} a_i-d+1.$$ Hence $\sum_{x_i\notin
\fp} a_i-d\leq -1$, so $k(\sum_{x_i\notin \fp} a_i-d)\leq -1$.
Therefore by \cite[Proposition 5.2]{HRV}, $\fp\in\Ass(S/J^k)$.
\end{proof}

\begin{Theorem}    \label{allweknow}
Let  $I=uJ$ be a  polymatroidal ideal
    of $S=K[x_1,\ldots,x_n]$, where $u$ is a monomial ideal and $J$ is a
    Veronese type ideal of $S$. Then $I$ has stable projective
    dimension.
 \end{Theorem}

 \begin{proof}  First, we show that if $I$ is a Veronese type ideal, then $I$ has stable projective dimension.
    Let $I=I_{(d;a_1,\ldots,a_n)}$ with $1\leq a_i\leq d$  for all $i$. If $\pd S/I=n$, then  there is  nothing to prove.
    Now let $\pd S/I=\sum_{i=1}^na_i-d+1$.
    Let $\fp\in V^*(I)$ with $\height\fp\geq \pd S/I$. We want to show that $\pd S/I=\pd S(\fp)/I(\fp)$. We may assume that
    $\fp=\fp_{\{1,\ldots,k\}}$. Therefore,   $n-k=\height\fp\geq \pd S/I$.
    We claim that  $a_j\leq d-\sum_{i=1}^ka_i$ for all $j\notin [k]$. We have $\sum_{i=1}^n a_i\leq n-k+d-1$, because  $\pd S/I\leq n-k$.
    Hence, since  $a_i\geq 1$ for all $i$, it follows that
    $$a_j+n-k-1\leq a_j+
    \sum_{\underset{i\neq j}{k+1\leq i\leq n}} a_i=\sum_{i=k+1}^n a_i\leq n-k+d-1-\sum_{i=1}^k a_i, $$
    for all $j\notin[k]$. This implies that  $a_j\leq d-\sum_{i=1}^ka_i$ for all $j\notin [k]$.
    Therefore, since $I(\fp)=I_{(d-\sum_{i=1}^ka_{i}; a_{k+1},\ldots,a_n)}$, it follows that
     $\pd S(\fp)/I(\fp)=\min\{n-k, \sum_{i=1}^na_i-d+1\}$, by Lemma \ref{knownformula}. Hence
     $\pd S(\fp)/I(\fp)=\sum_{i=1}^n a_i-d+1=\pd S/I$,  because $\sum_{i=1}^n a_i-d+1\leq
     n-k$.

    Now we  assume more generally  that  $I=uJ$ is a polymatroidal,
    where $J$ is a Veronese type ideal of $S$.   Let $\fp\in V^*(I)$ with
    $\height\fp\geq \pd S/I$. Then since  $\pd S/I=
    \pd S/J$,  it  follows that $\height\fp\geq \pd S/J$ and also  by Lemma \ref{veronese}, $\fp\in V^*(J)$. By the
    first part of the proof we know that $\pd S/J=\pd S(\fp)/J(\fp)$. Hence
    $\pd S/I=\pd S(\fp)/I(\fp)$, as desired.
\end{proof}

\begin{Example}\label{equi}
    The Veronese type  ideal $I=I_{(3;2,2,1)}\subset S=K[x_1,x_2,x_3]$  has stable
    projective dimension. Although
$I=(x_1^2,x_2)\cap(x_1,x_2^2)\cap(x_1,x_3)\cap(x_2,x_3)\cap(x_1^2,x_2^2,x_3)$
    is equidimensional, it is  not Cohen-Macaulay.
\end{Example}

Let $I$ be a {\it transversal polymatroidal ideal}, say,
$I=\fp_1\cdots \fp_r$ for monomial prime ideals
$\fp_1,\ldots,\fp_r$. As in \cite{HRV} we define the graph $G_I$
associated with $I$ as follows: the vertex set of $G_I$ is $[r]$ and
$\{i,j\}$ is an edge of $G_I$ if and only if $G(\fp_i)\sect
G(\fp_j)\neq \emptyset$.

The  projective dimension of a transversal polymatroidal ideal is
readable from its connected components:
\begin{Lemma}
    \label{almostknown}\cite[Theorem 4.12]{HRV}
    Let $I=\fp_1\cdots \fp_r$ be a transversal polymatroidal ideal, and let $H_1,\ldots,H_s$ be the connected components of $G_I$. Then
    \[
    \pd I=\sum_{j=1}^s(|\Union_{i\in H_j}G(\fp_i)|-1)= |\supp(I)|-s.
    \]
\end{Lemma}

\begin{Theorem}\label{tran}

 Let $I$ be a transversal polymatroidal ideal of $S=K[x_1,\ldots,x_n]$. Then $I$ has stable projective
    dimension  if and only if one of the following conditions is satisfied:
    \begin{anumbered}
        \item $I$ is a product of principal ideals.
        \item $I$ is the power of a monomial prime ideal.
        \item $G_I$ is connected and $I$ is fully supported.
    \end{anumbered}
\end{Theorem}

\begin{proof}
     Assume that $I$ has stable projective dimension.  Let $I=\fp_1\cdots \fp_r$  and  $H_1,\ldots,H_s$ be the connected components of $G_I$,
    and set $I_j=\prod_{i\in H_j}\fp_i$. We first  show that if
    $|\fp_i|\geq 2$ for each $i\in[r]$, then $G_I$ is connected and $I$ is fully supported or $I$ is the power of a monomial prime ideal.

    For each $j\in[s-1]$, we choose  $i_j\in H_j$, and set $\fp=\fp_{\{i_1,\ldots,i_{s-1}\}}$. Since $\Union_{i\in H_s} G(\fp_i)\subseteq \fp$,
    it follows that $\fp\in V^*(I)$. Moreover,   $\height\fp=n-(s-1)\geq \pd S/I=|\supp(I)|-(s-1)$,
    see Lemma \ref{almostknown}. Therefore,
    since $I$ has stable projective dimension,
    it follows that $\pd I=\pd I(\fp)$.

    Next we show that $\pd I_j(\fp)\leq |\supp(I_j)|-2$ for all $j\in[s-1]$. Indeed, if $I_j(\fp)=S(\fp)$  for some $j\in[s-1]$, then $\pd I_j(\fp)=0$,
    and the assertion is trivial because   $|\supp(I_j)|\geq 2$. On the other hand, if   $I_j(\fp)\neq S(\fp)$,
     then  by using Lemma \ref{almostknown}, we obtain
    $$\pd I_j(\fp)=|\supp(I_j(\fp))|-h\leq |\supp(I_j)|-1-h\leq|\supp(I_j)|-2,$$
    where $h\geq 1$ is the number of connected components of $G_{I_j(\fp)}$.
    So in each case we have that $\pd I_j(\fp)\leq |\supp(I_j)|-2$ for
    all $j\in[s-1]$.

    Since $I(\fp)=(\prod_{j=1}^{s}I_j(\fp))$, it follows that
      $\pd I(\fp)=\sum_{j=1}^{s-1}\pd I_j(\fp)+\pd I_s$ because $G(I_s)=G(I_s(\fp))$. Therefore, by using
    Lemma~\ref{almostknown}, we get
    $$\pd I=\pd I(\fp)\leq \sum_{j=1}^{s-1}(|\supp(I_j)|-2)+|\supp(I_s)|-1=\pd I-(s-1),$$
    which implies that  $s=1$. Hence $G_I$ is connected.

    Now, assuming  that  $I$  is not the power of a monomial prime ideal, we  show that $I$  is fully supported.
    Suppose to the  contrary  that  $|\supp(I)|<n$. Since $I$  is not the power of a monomial prime,  there exist  $1\leq j,k\leq r$ and
    $x_t\in G(\fp_j)\setminus  G(\fp_k)$. Let $\fp=\fp_{\{x_t\}}$. Then $\fp\in V^*(I)$ because
    $\fp_k\subset \fp$ and $\height(\fp)=n-1\geq\pd S/I$.

    Suppose  $I(\fp)=S(\fp)$, then $0=\pd I(\fp)=\pd I=|\supp(I)|-1\geq 1$, a contradiction.
    On the other hand,  if $I(\fp)\neq S(\fp)$,  then
    $\pd I(\fp)=|\supp(I(\fp))|-d$, where $d\geq 1$ is the number of connected components of $I(\fp)$.
    Hence $\pd I(\fp)\leq|\supp(I)|-1-d\leq|\supp(I)|-2=\pd I-1<\pd I$,
    which is  again a contraction. Thus we must have that $I$ is fully supported.

    In order to conclude our proof that $I$ must satisfy one of the
    conditions (a),(b) or (c), it remains to be shown that if there
    exists $l\in[r]$, such that $\fp_l$ is principal, then $I$ satisfies one of the  conditions (a) or (c). If
    $\fm\in\Ass(S/I)$, then $G_I$ is connected and $I$ is fully
    supported
    by   \cite[Theorem 4.3]{HRV}. Otherwise,  we show that
    $\fp_i$ is
    principal for all $i\in[r]$. If $\fp_d$ is not principal for some $d\in[r]$, then  there exists $x_t\in
   \fp_d\setminus \fp_l$, and   $\fp=\fp_{\{x_t\}}\in V^*(I)$, because
    $\fp_l\subset \fp$. Since  $\fm\notin
    \Ass(S/I)$, hence $\pd S/I\leq n-1=\height(\fp)$. We may assume that
    $\fp_d$ is one of the factors of the ideal $I_1$, as defined in the
    first part of the proof. Then $\pd I_1=|\supp(I_1)|-1\geq 1$. If
    $I_1(\fp)=S(\fp)$, then  $\pd I_1(\fp)<\pd I_1$, and  if $I_1(\fp)\neq
    S(\fp)$, then by Lemma \ref{almostknown}, $\pd
    I_1=|\supp(I_1)|-1>|\supp(I_1(\fp))|-1\geq\pd I_1(\fp)$. So in each case
    $\pd I_1>\pd I_1(\fp)$. Hence since  $\pd I_j\geq\pd I_j(\fp)$ for all
    $j\in[s]$, it follows that $\pd I=\sum_{i=1}^s\pd
    I_i>\sum_{i=1}^s\pd I_i(\fp)=\pd I(\fp)$, which contradicts our
    assumption that $I$ has stable projective dimension.

    Conversely, it is obvious that $I$ has stable projective dimension in the cases (a) and (b). Now, assume
    condition (c) is satisfied.  Then  Lemma~\ref{almostknown} implies that $\pd I=n-1$,  and hence $\depth S/I=0$, so $I$ has stable
     projective dimension.
     \end{proof}

The following examples, show that  stability property of projective
dimension is not inherited by  radical and power.

\begin{Example}
    \begin{anumbered}
        \item Let $I=(x_1,x_2,x_3)(x_1,x_4)\subset
        S=K[x_1,\ldots,x_4]$. Since $I$ is connected and fully supported transversal polymatroidal ideal,
        it follows by Theorem \ref{tran} that $I$   has stable projective dimension, but
        $\sqrt{I}=(x_1,x_2x_4,x_3x_4)$ is not Cohen-Macaulay, so does not
        have stable projective dimension.
        \item Let $I=(x_1x_2,x_2x_3,x_3x_4)$. $S/I$ is Cohen-Macaulay
        and $\Ass(S/I^2)=\Min(I^2)$, but $S/I^2$ is not Cohen-Macaulay.
        So by Proposition \ref{min},
        $S/I^2$ does not have stable projective dimension.
    \end{anumbered}
\end{Example}

\section*{Acknowledgments}
The authors would like to thank Professor J\"{u}rgen Herzog for
fruitful discussions and useful comments regarding this paper.

\end{document}